\documentclass[10pt,reqno]{amsart}
\usepackage[utf8]{inputenc}
\usepackage[UKenglish]{babel}
\usepackage{amsmath, amssymb, amsthm}
\usepackage{mathrsfs}
\usepackage{enumitem}
\usepackage{graphicx}
\usepackage{tikz-cd}
\usepackage{xcolor}
\numberwithin{equation}{section} 

\usepackage{diagbox} 
\usepackage{tensor}  

\usepackage{hyperref} 

\usepackage[alphabetic]{amsrefs} 

\theoremstyle{plain}
\newtheorem{thm}{Theorem}[section]
\newtheorem{prop}[thm]{Proposition}
\newtheorem{lemma}[thm]{Lemma}

\theoremstyle{definition}
\newtheorem{defin}[thm]{Definition}

\newtheorem{rmk}[thm]{Remark}

\allowdisplaybreaks

\newcommand{\R}{\mathbb{R}}

\newcommand{\norm}[1]{\left\lVert#1\right\rVert}

\newcommand{\eps}{\varepsilon}

\DeclareMathOperator{\vol}{Vol}

\newcommand{\MP}{\mathcal{MP}}
\newcommand{\A}{\mathbb{A}}
\newcommand{\dd}{\mathbf{d}}

\newcommand{\LL}{\mathbf{L}}
\newcommand{\cal}{\mathcal{L}}

\def\XXint#1#2#3{{\setbox0=\hbox{$#1{#2#3}{\int}$ }
\vcenter{\hbox{$#2#3$ }}\kern-.6\wd0}}

\usepackage{scalerel,stackengine}
\stackMath
\newcommand\hhat[1]{%
\savestack{\tmpbox}{\stretchto{%
  \scaleto{%
    \scalerel*[\widthof{\ensuremath{#1}}]{\kern.1pt\mathchar"0362\kern.1pt}%
    {\rule{0ex}{\textheight}}
  }{\textheight}%
}{2.4ex}}%
\stackon[-6.9pt]{#1}{\tmpbox}%
}

\title[The linearization of the boundary rigidity problem for $\MP$-systems]{The linearization of the boundary rigidity problem for MP-systems and generic local boundary rigidity}
\author{Sebastián Muñoz-Thon}

\address{Department of Mathematics, Purdue University, West Lafayette, IN 47907.}
\email{smunozth@purdue.edu}

\begin{document}

\begin{abstract}
We consider an $\MP$-system, that is, a compact Riemannian manifold with boundary, endowed with a magnetic field and a potential. On simple $\MP$-systems, we study the $\MP$-ray transform in order to obtain new boundary rigidity results for $\MP$-systems. We show that there is an explicit relation between the $\MP$-ray transform and the magnetic one, which allow us to apply results from \cite{dpsu} to our case. Regarding rigidity, we show that there exists a generic set $\mathcal{G}^{m}$ of simple $\MP$-systems, which is open and dense, such that any two $\MP$-systems close to an element in it and having the same boundary action function, must be $k$-gauge equivalent.
\end{abstract}

\maketitle

\section{Introduction}

\subsection{Previous results}

In geometric inverse problems, the boundary rigidity problem is a classical question that ask to what extent one can recover the metric $g$ by knowing the boundary distance function $d_{g}|_{\partial M \times \partial M}$. Here $(M,g)$ is a smooth ($C^{\infty}$) compact Riemannian manifold with smooth boundary. Instead of attacking the problem directly, one could try to analyze its linearized version. This is known as the X-ray tomography problem, in which one try to recover a function (or more generally, tensors) from the knowledge of its integrals along geodesics. This operator, the X-ray transform, appears naturally in medical imaging, and it is from this context that it gets its name. References about this operator and some generalizations are \cite{sha} and \cite{psu}.

One of the generalization in which we are interested is to study the X-ray transform over more general curves and obtain information about the boundary rigidity problem involving that curves. In \cite{dpsu}, the authors study the magnetic ray transform, which arises as the linearization of the boundary rigidity problem on simple magnetic systems $(M,g,\alpha)$, where $(M,g)$ is a smooth compact Riemannian manifold with smooth boundary and $\alpha$ is a 1-form. In that work, the authors study the X-ray transform and its normal operator. They solved (up to a natural gauge) the boundary rigidity problem, on a conformal class, for analytic magnetic systems, for surfaces, and for metrics close a generic set of metrics. They also solve the linear problem for 1-tensors and 2-tensors. The magnetic ray transform is also studied in \cite{ainsworth} and \cite{zhou}.

In this paper we focus on the $\MP$ case. An $\MP$-system consist of a smooth compact Riemannian manifold with smooth boundary $(M,g)$, a closed 2-form $\Omega$, and a smooth function $U$. The curves in this case describe the motion of a particle on a Riemannian manifold under the influence of a magnetic field represented by $\Omega$, and a potential field represented by the function $U$. This kind of systems appears in mechanics, see \cite{kozlov}, \cite{arnold}, \cite{akn}, \cite{cggmw}, \cite{maraner}. They also appear when one studies geodesics on Lorentzian manifolds endowed with stationary metrics \cite{germinario}, \cite{bg}, \cite{plamen}, in inverse problems for the acoustic wave equation from phaseless measurements \cite{iw}, and in inverse problems in transport equations with external forces on Euclidean domains \cite{lz}. We would like to mention that the X-ray transform has also been studied for general curves, see \cite{fsu}, \cite{uv}, \cite{ad}, and \cite{zhang}.

Given an $\MP$-system $(M,g,\Omega,U)$, the magnetic field $\Omega$ induces a map $Y \colon TM \to TM$ given by
\[ \Omega_{x}(u,v)=(Y_{x}u,v)_{g}, \]
where $u,v \in T_{x}M$. This map is usually called \emph{Lorentz force} associated to the magnetic field $\Omega$. For simple $\MP$-systems, $C^{2}$ curves $\sigma \colon [a,b] \to M$ that satisfy
\begin{equation} \label{eq:mp-geo}
    \nabla_{\dot{\sigma}} \dot{\sigma}=Y(\dot{\sigma})-\nabla U(\sigma),
\end{equation}
are called \emph{$\MP$-geodesics}. Here, $\nabla$ is the Levi-Civita connection associated to the metric $g$. Equation \eqref{eq:mp-geo} defines a flow , called the \emph{$\MP$-flow}, and given by
\[ \phi_{t}(x,v)=(\sigma(t),\dot{\sigma}(t)),  \]
where $\sigma$ solves \eqref{eq:mp-geo} and $\sigma(0)=x$, $\dot{\sigma}(0)=v$. See \cite{mt23}*{Lemma A.1} for other interpretations of the $\MP$-flow. For the $\MP$-flow the energy $E \colon TM \to \R$ given by $E(x,v)=\frac{1}{2}|v|_{g}^{2}+U(x)$ is an integral of motion. Indeed, for $\sigma$ satisfying \eqref{eq:mp-geo}, we have
\[ \frac{d}{dt}E(\sigma(t),\dot{\sigma}(t))=( \nabla_{\dot{\sigma}}\dot{\sigma},\dot{\sigma})_{g}+(\nabla U,\dot{\sigma})=( Y \dot{\sigma},\dot{\sigma})_{g}=\Omega(\dot{\sigma},\dot{\sigma})=0. \]
Then, the energy is constant along $\MP$-geodesics. It has been shown that $\MP$-geodesics minimize the \emph{time free action} of energy $k$ (see \cite{az}*{Appendix A.1})
\[ \A(\sigma)=\frac{1}{2} \int_{0}^{T}|\dot{\sigma}|_{g}^{2}dt+kT-\int_{0}^{T}(\alpha(\sigma(t),\dot{\sigma}(t))+U(\sigma(t)) )dt, \]
where 
\[ \sigma \in \mathcal{C}(x,y)=\{\sigma \in AC([0,T],M): \sigma(0)=x \text{ and } \sigma(T)=y\}, \]
so that the \emph{Ma\~n\'e action potential} (of energy $k$) is well defined
\begin{equation} \label{eq:mane_act}
    \A(x,y)=\inf_{\gamma \in \mathcal{C}(x,y)}\A(\gamma).
\end{equation}
The ``action'' terminology comes form physics and Lagrangian flows, see \cite{am}, \cite{arnoldclas}, \cite{ci}, \cite{paternain}, \cite{mazzucchelli}.

The \emph{boundary rigidity problem for $\MP$-systems} asks to what extent one can recover the metric $g$, the magnetic field $\Omega$ (or $\alpha$), and the potential $U$, by knowing the boundary action function $\mathbb{A}_{g,\alpha,U}|_{\partial M \times \partial M}$ (of energy $k$). In the flat case, these problems were studied in \cite{jo2007}. For Riemannian manifolds, these problems were studied in \cite{az} \cite{mt23}. In the former one, the authors prove that for simple $\MP$-systems, the knowledge of the boundary action function for two energy levels allows to recover the system $(g,\alpha,U)$ up to a gauge, in three cases: working on the same conformal class, for analytic $\MP$-systems, and working on surfaces. In \cite{mt23}, the author prove the same results but only assuming the knowledge of the boundary action function for one energy level, under a more general gauge, see Definition \ref{defin:gauge}.

In this work, we study the linearization of the boundary rigidity problem for $\MP$-systems. We study this problem because is interesting in its own, and in order to obtain new rigidity results. Regarding the linearization, we have the following:
\begin{itemize}
    \item s-injectivity of the $\MP$-ray transform for analytic systems (Theorem \ref{thm:an_inj});
    \item solve the linear problem for 1-tensors (Theorem \ref{thm:linear_1tensors});
    \item solve the linear problem for 2-tensors (Theorem \ref{thm:linear_2tensors});
    \item for $m$ big enough, we show that there is an open dense set in the $C^{m}$ topology in which the $\MP$-ray transform is s-injective (Theorem \ref{thm:gen_sinj}).
\end{itemize}
See Definition \ref{defin:s_inj} for the definition of s-injectivity.
Using this generic set, we obtain a local generic rigidity result: we prove that two systems with the same boundary action function at one energy level, both close to a system in the generic set, are $k$-gauge equivalent.

The proof of the results rely on the relation between $\MP$-systems and magnetic ones. We obtain a relation between the $\MP$-ray transform using the relation between $\MP$-geodesics and magnetic ones (Lemma \ref{lemma:basics_mp}), and the fact that $\MP$-geodesics has constant energy, which gives a relation between 2-tensors and functions. We also obtain a relation between the notion of s-injectivity for the $\MP$-ray transform and for the magnetic ray transform (Proposition \ref{prop:s_inj}). This reduce the study of potentials on the $\MP$ sense to the magnetic ones. Hence, we are able to apply results in \cite{dpsu} to obtain new information about the linear problem for $\MP$-systems. To obtain the generic results, we show that if two $\MP$-systems are close in the $C^{m}$ topology, then their reductions are also close, and we apply the generic results known for magnetic cases.

\subsection{Structure of the paper}

In Section \ref{sec:review} we briefly summarize some facts about $\MP$-systems and they relation with magnetic systems that would be used in this work. In Section \ref{sec:ray_transform} we define the $\MP$-ray transform, we show that is the linearization of the boundary rigidity problem, and explore its relation with the magnetic ray transform. In Section \ref{sec:potential} we study the potential parts and their relation with the potential parts of the magnetic ray transform. In Section \ref{sec:s_inj_an} we prove the s-injectivity of the $\MP$-ray transform for simple analytic systems. In Section \ref{sec:linear} we study the linear problem. In Subsection \ref{subsec:gen_s_inj} we use results from Section \ref{sec:s_inj_an} to obtain a generic set in which the $\MP$-ray transform is s-injective. In Subsection \ref{subsec:gen_rig} we obtain the generic local result. Finally, in the appendix we give a Santal\'o's formula for $\MP$-systems.

\subsection*{Acknowledgments}

The author would like to thank Plamen Stefanov for suggesting this problem, for helpful discussions about magnetic systems, and for helpful comments on a previous version of this manuscript. The author would like to thank Gunther Uhlmann for helpful suggestions on a preliminary version for this work. The author was partly supported by NSF Grant DMS-2154489.

\section{Review of MP-systems} \label{sec:review}

In this section we summarize important properties of $\MP$-systems that will be used in this work.

Recall that an $\MP$-system is consist of tuple $(M,g,\Omega,U)$, where $M$ is a compact smooth manifold with smooth boundary, $g$ is a Riemannian metric, $\Omega$ is a closed 2-form and $U$ is a smooth function. Since we are going to work on a fixed manifold $M$, we will going to refer to the triple $(g,\Omega,U)$ as an $\MP$-system. 

As was mentioned in the introduction, $\MP$-geodesics has constant energy. Given $k \in \R$, we define $S^{k}M:=\{E=k\}$. We will assume always that $k>\max_{M} U$ (so that $S^{k}M$ is a non-empty level set, see also \cite{az}*{Appendix A} and \cite{ci}*{Chapter 2}). Let $\nu(x)$ be the inward unit vector normal to $\partial M$ at $x$, and set
\[
\partial_{\pm} S^{k}M:=\{(x, v) \in S^{k}M: x \in \partial M, \pm (v, \nu(x))_{g} \geq 0\}.
\]
For $x \in M$, the $\MP$-exponential map at $x$ is the map $\exp _x^{\MP} \colon T_{x} M \to M$ given by
\[
\exp _x^{\MP}(t v)=\pi \circ \phi_t(v),
\]
where $t \geq 0$, $v \in S_{x}^{k} M$, and $\pi \colon TM \to M$ is the canonical projection. 

Let $\Lambda$ denotes the second fundamental form of $\partial M$. Consider a manifold $M_1$ such that $M_1^{\text {int }} \supset M$. Extend $g$, $\Omega$ and $U$ to $M_1$ smoothly, preserving the former notation for extensions. $M$ is said to be \emph{$\MP$-convex} at $x \in \partial M$ if there is a neighborhood $O$ of $x$ in $M_1$ such that all $\MP$-geodesics of constant energy $k$ in $O$, passing through $x$ and tangent to $\partial M$ at $x$, lie in $M_1 \setminus M^{\text {int }}$. If, in addition, these geodesics do not intersect $M$ except for $x$, we say that $M$ is \emph{strictly $\MP$-convex} at $x$. By \cite{az}*{Lemma~A.2}, strictly $\MP$-convexity at $x \in \partial M$ implies
\[
\Lambda(x, v)>\langle Y_x(v), \nu(x)\rangle-d_x U(\nu(x))
\]
for all $(x, v) \in S^{k}(\partial M)$. 

\begin{defin} \label{defin:simple}
We say that $M$ is ($\MP$) \emph{simple} with respect to $(g, \Omega, U)$ if $\partial M$ is strictly $\MP$-convex and the $\MP$-exponential map $\exp_{x}^{\MP} \colon (\exp_x^{\MP})^{-1}(M) \to M$ is a diffeomorphism for every $x \in M$.    
\end{defin}

In this case, $M$ is diffeomorphic to the unit ball of $\R^n$. Hence, Poincaré's lemma implies that $\Omega$ is exact, that is, there exist a 1-form $\alpha$ such that $\Omega=d\alpha$, and we call $\alpha$ to be the magnetic potential.

Henceforth we call $(g, \alpha, U)$ a simple $\MP$-system on $M$. We will also say that $(M, g, \alpha, U)$ is a simple $\MP$-system. When $\alpha=0$ and $U=0$, these notions coincide with the usual notion of simple Riemannian manifold, see \cite{psu}*{Section~3.8}.

For $(x, v) \in \partial_{+} S^{k}M$, let $\tau(p, v)$ be the time when the $\MP$-geodesic $\sigma$, such that $\sigma(0)=p$, $\dot{\sigma}(0)=v$, exits. By \cite{az}*{Lemma~A.3} we have that for a simple $\MP$-system, the function $\tau \colon \partial_{+} S^{k}M \to \R$ is smooth.

\begin{defin}
Given a simple $\MP$-system $(g,\alpha,U)$ of energy $k$, we associate to it the magnetic system $(2(k-U)g,\alpha)$ of energy $\frac{1}{2}$, which we call \emph{reduced magnetic system}.    
\end{defin}

The interplay between an $\MP$-system an its magnetic reduction was studied in detail in \cite{az}. We now state some results that are useful for this work.

\begin{lemma}[\cite{az}*{Proposition~1, 2, 3}] \label{lemma:basics_mp}
Let $(g,\alpha,U)$ be $\MP$-system with energy $k$ an let $(G,\alpha):=(2(k-U)g,\alpha)$ be its reduction to a magnetic system of energy $\frac{1}{2}$.
\begin{enumerate}
    \item If $k>\max_{x \in M}U$ and $\sigma$ is an $\MP$-geodesic of energy $k$, then there exists a reparametrization of $\sigma$ that is a unit speed magnetic geodesic for the reduced system.
    \item $(g,\alpha,U)$ is simple (in the $\MP$-sense) if and only if $(G,\alpha)$ is simple (in the magnetic sense).
    \item Let $\A$ be the Ma\~n\'e's action potential of energy $k$ for $(g, \alpha, U)$ and $\mathbb{A}_G$ be the Ma\~n\'e's action potential of energy 1/2 for the simple magnetic system $(G, \alpha)$, then $\A|_{\partial M \times \partial M}=\A_{G}|_{\partial M \times \partial M}$.
\end{enumerate}
\end{lemma}

The definition of simple magnetic systems and the definition of the Ma\~n\'e's potential for magnetic systems, are the same as above with $U=0$.

The new parameter such that $\gamma(s)=\sigma(t(s))$ is a magnetic unit speed geodesic is given by
\[ s(t)=\int_{0}^{t}2(k-U(\sigma))dt. \]
Part (1) on Lemma \ref{lemma:basics_mp} is known as Jacobi--Maupertuis' principle.

We also recall the notion of gauge for $\MP$-system defined by the author.

\begin{defin} \label{defin:gauge}
We say that two $\mathcal{M P}$-systems $(g, \alpha, U)$ and $(g', \alpha', U')$ are $k$-\emph{gauge equivalent} if there is a diffeomorphism $f\colon M \to M$ with $f|_{\partial M}=id_{\partial M}$, a smooth function $\varphi \colon M \to \R$ with $\varphi|_{\partial M}=0$, and a strictly positive function $\mu \in C^{\infty}(M)$, such that $g'=\frac{1}{\mu}f^{*}g$, $\alpha'=f^* \alpha+d \varphi$ and $U'=\mu (f^{*}U-k)+k$.
\end{defin}

Finally, we would like to mention a relation between the notion of $k$-gauge equivalence for $\MP$ and the notion of gague equivalence of their reduced systems.

\begin{lemma}[\cite{mt23}*{Lemma 4.3}] \label{lemma:mp_red}
Let $(g,\alpha,U)$ and $(g',\alpha',U')$ be two $k$-gauge equivalent $\MP$-system. Then, their reduced magnetic systems are magnetic gauge equivalent at energy $\frac{1}{2}$. Reciprocally, if $(G,\alpha)$ and $(G,\alpha')$ are gauge equivalent magnetic systems (of energy $\frac{1}{2}$) given by the reduction of the $\MP$-systems $(g,\alpha,U)$ and $(g',\alpha',U')$, then the $\MP$-systems are $k$-gauge equivalent.
\end{lemma}

\section{The MP-ray transform} \label{sec:ray_transform}

In this section we define the $\MP$-ray transform. We show that, as in the Riemannian and the magnetic case, arises as the linearization of the boundary action function. Finally, we show how is related with the magnetic ray transform of the reduced magnetic system.

\subsection{Definitions}

\begin{defin}
Let $(M,g,\alpha,U)$ be a simple $\MP$-system. For $f \in C^{\infty}(S^{k}M,\R)$. We define the \emph{$\MP$-ray transform} of $f$ by
\[ If(\sigma)=\int_{\sigma}f:=\int_{0}^{T}f(\sigma(t),\dot{\sigma}(t))dt,\]
where $\sigma \colon [0,T] \to M$ is any geodesic of energy $k$ with $\sigma(0),\sigma(T) \in \partial M$.
\end{defin}

This operator has been studied in \cite{iw}*{Appendix~A} for functions. In the next section we will study a more general operator, that is, the $\MP$-ray transform on tensors (see Definition \ref{def:mp_ray}).

Assuming that $\MP$-geodesics of energy $k$ are parametrized by $\partial_{+}S^{k}M$, we obtain a map $I \colon C^{\infty}(SM) \to C(\partial_{+}S^{k}M)$ given by
\[ If(x,v)=\int_{0}^{\tau(x,v)}f(\phi_{t}(x,v))dt, \qquad (x,v) \in \partial_{+}S^{k}M. \]

In the space of real valued function on $\partial_{+}S^{k}M$, we define the norm
\[ \norm{f}^{2}=\int_{\partial_{+}S^{k}M}|f|^{2}d\mu_{k}, \]
and the corresponding inner product. Here $d\mu_{k}(x,v)=(v,\nu(x))_{g} d\Sigma_{k}^{2n-2}$, where $\Sigma_{k}^{2n-2}$ is the Liouville measure on $\partial_{+}S^{k}M$. See Appendix \ref{sec:app}. The corresponding Hilbert space will be denoted by $L_{\mu}^{2}(\partial_{+}SM)$.

\begin{lemma}
The operator $I$ extends to a bounded operator
\[ I \colon L^{2}(SM) \to L_{\mu}^{2}(\partial_{+}SM). \]
\end{lemma}

\begin{proof}
By Cauchy--Schwarz inequality, for $(x,v) \in \partial_{+}S^{k}M$ we have
\[ (If (x,v))^{2} \leq C \int_{0}^{\tau(x,v)}f(\phi_{t}(x,v))^{2}dt,\]
where 
\[ C=\max_{(x,v) \in \partial_{+}S^{k}M}\tau(x,v). \]
Let $P(x)=2(k-U(x))$. Then, applying Santaló's formula (Lemma \ref{prop:santalo}) we obtain
\begin{align*}
    &\int_{\partial_{+}S^{k}M}(If)^{2}d\mu_{k} \\
    \leq & C \int_{\partial_{+}S^{k}M} \int_{0}^{\tau(x,v)} f(\phi_{t}(x,v))^{2}dt d\mu_{k} \\
    =& C \int_{\partial_{+}S^{k}M} \left( \int_{0}^{\tau(x,v)} f(\phi_{t}(x,v))^{2} P^{\frac{1}{2}}(\sigma(t)) P^{-\frac{1}{2}}(\sigma(t)) dt \right) P(x) P^{-1}(x)d\mu_{k} \\
    \leq & \tilde{C} \int_{\partial_{+}S^{k}M} \left( \int_{0}^{\tau(x,v)} f(\phi_{t}(x,v))^{2} P^{\frac{1}{2}}(\sigma(t))dt \right) P^{-1}(x)d\mu_{k} \\
    =&\tilde{C} \int_{S^{k}M}f^{2}d\Sigma^{2n-1},
\end{align*}
where
\[ \tilde{C}=C\max_{x \in M}P^{\frac{1}{2}}(x). \]
\end{proof}

\subsection{Linearization}

We will show that the $\MP$-ray transform is the linearization of the boundary action function.

Let $(g,\alpha,U)$ be a simple $\MP$-system on $M$. Take $\eps>0$ small enough so that every $\MP$-system $(g+h,\alpha+\beta,U+V)$ satisfying
\begin{equation} \label{eq:small_pert}
    \norm{h}_{C^{2}} \leq \eps,\qquad \norm{\beta}_{C^{1}} \leq \eps \qquad \norm{V}_{C^{1}}\leq \eps,
\end{equation}
is simple. Given $h$, $\beta$, and $V$ satisfying \eqref{eq:small_pert}, consider the 1-parameter family of $\MP$-simple systems $(g^{s},\alpha^{s},U^{s})$, where 
\[ g^{s}=g+sh, \qquad \alpha^{s}=\alpha+s\beta, \qquad U^{s}=U+sV, \]
and $s \in [0,1]$.

\begin{lemma} \label{lemma:lin}
For $x, y \in \partial M$,
\[
\frac{d \A_{g^s, \alpha^s,U^{s}} (x, y)}{d s}= \frac{1}{2} \int_{\sigma_s}\langle h, \dot{\sigma}_s^2\rangle-\int_{\sigma_s} \beta-\int_{\sigma_{s}}V,
\]
where $\sigma_s \colon [0,T_{s}] \to M$ is the $\MP$-geodesic with constant energy $k$ from $x$ to $y$ with respect to $(g^s, \alpha^s,U^{s})$.    
\end{lemma}

Our proof is based in the proof of a similar result \cite{dpsu}*{Lemma~3.1}.

\begin{proof} Define
\[
\varphi(s, \tau):=\A_{g^\tau, \alpha^\tau,U^{\tau}}(\sigma_s)=\frac{1}{2} \int_0^{T_s}|\dot{\sigma}_s(t)|_{g^\tau}^2 d t+k T_s-\int_{\sigma_s} \alpha^\tau-\int_{\sigma_{s}}U^{\tau}.
\]
Then
\[
\frac{d \A_{g^{s}, \alpha^{s},U^{s}}(x, y)}{d s}=\frac{\partial \varphi}{\partial s}(s, s)+\frac{\partial \varphi}{\partial \tau}(s, s) .
\]
By \cite{az}*{Lemma~A.1}, $\MP$-geodesic with energy $k$ minimize the time free action. Therefore, for every fixed $\tau$, $\A_{g^{\tau},\alpha^{\tau},U^{\tau}}(\gamma_{s})$ has a minimum at $s=\tau$. Thus, 
\[ \frac{\partial \varphi}{\partial s}(s, s)=0. \]
On the other hand, 
\begin{align*}
    \frac{\partial \varphi}{\partial \tau}=&\frac{1}{2}\int_{0}^{T_{s}} \frac{\partial}{\partial \tau}(g_{ij}(\sigma_{s}(t))+\tau h_{ij}(\sigma_{s}(t)))\dot{\sigma}_{s}^{i}(t)\dot{\sigma}_{s}^{j}(t) dt \\
    &-\int_{0}^{T_{s}} \frac{\partial}{\partial \tau}(\alpha_{i}(\sigma_{s}(t))+\tau \beta_{i}(\sigma_{s}(t)))\dot{\sigma}_{s}^{i}(t) dt \\
    &-\int_{0}^{T_{s}}\frac{\partial}{\partial \tau} (U(\sigma_{s}(t))+\tau V(\sigma_{s}(t)))dt \\
    =&\frac{1}{2}\int_{0}^{T_{s}} h_{ij}(\sigma_{s}(t))\dot{\sigma}_{s}^{i}(t)\dot{\sigma}_{s}^{j}(t) dt -\int_{0}^{T_{s}} \beta_{i}(\sigma(t))\dot{\sigma}_{s}^{i}(t) dt-\int_{0}^{T_{s}}V(\sigma_{s}(t))dt.
\end{align*}
\end{proof}

In view of Lemma \ref{lemma:lin}, we are interested in $I$ applied to functions of the form 
\begin{equation} \label{eq:poly}
    \psi(x,v)=h_{ij}(x)v^{i}v^{j}+\beta_{i}(x)v^{i}+V(x).
\end{equation}
This motivates the following definition

\begin{defin} \label{def:mp_ray}
Let $\mathbf{f}=[h,\beta,V] \in \mathbf{L}^{2}(M)$. For $(x,v) \in \partial_{+}S^{k}M$ we define the $\MP$\emph{-ray transform} by
\begin{equation} \label{eq:def_I}
    I[h,\beta,V](x,v)=\int_{0}^{\tau(x,v)} \left\lbrace h_{ij}(\sigma(t))\dot{\sigma}^{i}(t)\dot{\sigma}^{j}(t)+\beta_{i}(\sigma(t))\dot{\sigma}^{i}+V(\sigma(t)) \right\rbrace dt,
\end{equation}
where $\sigma \in S^{k}M$ is the only $\MP$-geodesic with $\sigma(0)=x$ and $\dot{\sigma}(0)=v$. 
\end{defin}

Here, $\mathbf{L}^{2}(M)$ is the space of square integrable ordered triples $\mathbf{f}=[h,\beta,V]$, where $h$ is a symmetric 2-tensor, $\beta$ is a 1-form and $V$ is a smooth function, endowed with the norm
\begin{equation} \label{eq:l2}
    \norm{\mathbf{f}}_{\mathbf{L}^{2}(M)}^{2}=\int_{M} \left\lbrace |h|_{g}^{2}+|\beta|_{g}^{2}+|V|^{2} \right\rbrace d\vol.
\end{equation}
We also consider the space $\mathcal{L}^{2}(M)$ of square integrable ordered triples $\mathbf{w}=[v,\varphi,\eta]$ (where $v$ is a vector field, and $\varphi$ and $\eta$ are smooth functions), endowed with the norm
\begin{equation} \label{eq:call2}
    \norm{\mathbf{w}}_{\mathcal{L}^{2}(M)}^{2}=\int_{M} (|v|_{g}+|\varphi|^{2}+|\eta|^{2})d\vol.
\end{equation}
In a similar way we define $\mathcal{H}_{0}^{1}(M)$. Recall that given an $\MP$-system of energy $k$, we have its reduction $(G,\alpha)$ where $G=2(k-\alpha)g$. We will consider the spaces associated to this magnetic systems, and we will use the same notation as before, but with the subscript $\mathcal{M}$. We will consider
\[ \LL^{2}_{\mathcal{M}}(M), \quad \cal^{2}_{\mathcal{M}}(M), \quad \mathcal{H}_{0,\mathcal{M}}^{1}(M). \]

\subsection{Relation with the magnetic ray transform}

As it known, there is a relation between boundary actions of $\MP$-systems and the boundary action of their reductions (see Lemma \ref{lemma:basics_mp}), we would like to find a relation between the $\MP$-ray transform and the magnetic ray transform (see \cite{dpsu}*{Equation 3.14} for the definition). To obtain this relation, we define 
\begin{align*}
    \Phi \colon \LL^{2}(M) & \to \LL_{\mathcal{M}}^{2}(M), \\
    [h,\beta,V] & \mapsto [2(k-U)h+gV,\beta].
\end{align*}
Is easy to see that $\Phi$ is a linear map. Furthermore, we have the following.

\begin{lemma} \label{lemma:Phi}
    $\Phi$ is a surjective linear map with kernel
    \begin{equation} \label{eq:ker_phi}
        \ker \Phi=\{ [h,\beta,V] \in \LL^{2}(M): h=-\eta g,\, \beta=0,\, V=2\eta(k-U), \, \eta \in L^{2}(M)\}.
    \end{equation}
\end{lemma}

\begin{proof}
    Let $[h,\beta] \in \LL_{\mathcal{M}}^{2}(M)$. Then, $\mathbf{f}=[\frac{1}{2(k-U)}h,\beta,0]$ satisfies $\Phi(\mathbf{f})=[h,\beta]$. This shows that the map is surjective.
    
    Now suppose that $\Phi(\mathbf{f})=0$, where $\mathbf{f}=[h,\beta,V] \in \LL^{2}(M)$. Then, $\beta=0$ and 
    \[ 2(k-U)h+gV=0, \]
    that is, $h=-\frac{g}{2(k-U)}V$. So, $\mathbf{f}=[-\frac{g}{2(k-U)}V,0,V]$. The reciprocal is trivial.
\end{proof}

We also make the following observation.

\begin{lemma} \label{lemma:i_phi}
    Let $\mathbf{f} \in \ker \Phi$, then $I\mathbf{f}=0$.
\end{lemma}

\begin{proof}
    Since $\mathbf{f} \in \ker \Phi$, we can write $\mathbf{f}=\eta[-g,0,2(k-U)]$. For $(x,v) \in \partial_{+}S^{k}M$, let $\sigma$ be the unique geodesic with energy $k$ such that $(\sigma(0),\dot{\sigma}(0))=(x,v)$. Then, 
    \begin{align*}
        I\mathbf{f}(x,v)&=\int_{0}^{\tau(x,v)} \{-\eta(\sigma(t))g_{ij}(\sigma(t)) \dot{\sigma}^{i}(t) \dot{\sigma}^{j}(t)+2\eta(k-U(\sigma(t))) \} dt \\
        &=\int_{0}^{\tau(x,v)} \eta(\sigma(t))\{-g_{ij}(\sigma(t)) \dot{\sigma}^{i}(t) \dot{\sigma}^{j}(t)+2(k-U(\sigma(t))) \} dt \\
        &=\int_{0}^{\tau(x,v)} \eta(\sigma(t))0 dt \\
        &=0.
    \end{align*}
\end{proof}

We will write $I$ for the $\MP$-ray transform over an $\MP$-geodesic $\sigma$, and $I_{\mathcal{M}}$ to denote the magnetic ray transform of its reduced magnetic system over the magnetic geodesic $\gamma(s)=\sigma(t(s))$. As was done in \cite{iw}*{Lemma A.1} for functions, we have the following relation between these transforms for tensors.

\begin{prop} \label{prop:rel_mpray_magray}
Let $\mathbf{f} \in \LL^{2}(M)$. For $(x,v) \in \partial_{+}S^{k}M$ we have
\[ I\mathbf{f}(x,v)=I_{\mathcal{M}}\Phi(\mathbf{f}) \left( x, \frac{v}{2(k-U(x))} \right). \]
\end{prop}

\begin{proof}
Let $\sigma=\sigma(t)$ be an $\MP$-geodesic with $(\sigma(0),\dot{\sigma}(0))=(x,b)$. Since $E(\sigma,\dot{\sigma})=k$, then     
\[ k=\frac{1}{2}g_{ij}(\sigma)\dot{\sigma}^{i}\dot{\sigma}^{j}+U(\sigma). \]
So, the $\MP$-ray transform takes the form
\begin{equation} \label{eq:mag}
    I[h,\beta,V]=\int_{\sigma} \left\langle h+\frac{g}{2(k-U)}V,\dot{\sigma}^{2} \right\rangle+\int_{\sigma}\beta.
\end{equation}
From Lemma \ref{lemma:basics_mp}, we know that $\gamma(s)=\sigma(t(s))$ is a magnetic geodesic of speed one of the reduced system $(G,\alpha)$ with
\[ \frac{d\gamma}{ds}=\frac{d \sigma}{d t} \frac{dt}{ds}=\frac{d\sigma}{dt} \frac{1}{2(k-U)}. \]
Hence,
\begin{align*}
    I\mathbf{f}(x,v)&=\int_{\sigma} \left\langle h+\frac{g}{2(k-U)}V,\dot{\sigma}^{2} \right \rangle +\int_{\sigma} \beta \\
    &=\int_{\gamma}  2(k-U) \left\langle h+\frac{g}{2(k-U)}V,\dot{\gamma}^{2} \right \rangle +\int_{\gamma} \beta \\
    &=I_{\mathcal{M}} \Phi(\mathbf{f})\left( x, \frac{v}{2(k-U(x))} \right).
\end{align*}
\end{proof}

This can be thought as a linearized version of Lemma \ref{lemma:basics_mp} part (3).

\section{Potential triples} \label{sec:potential}

In this section we define the notion of potential for the $\MP$-tray transform, and we study its relation with the potentials of $I_{\mathcal{M}}$. 

\subsection{The definition of potential for the MP-ray transform}

One option to understand potentials it to look at polynomials that are on the image of the generator of the flow. Recall that the generator of the $\MP$-flow is given by
\[ G_{\MP}(x,v)=G_{\mu}(x,v)-g^{ij}\partial_{x^{i}} U(x) \frac{\partial}{\partial v^{j}}, \]
where $G_{\mu}$ is the generator of the magnetic flow. Explicitly, 
\[ G_{\MP}(x,v)=v^{i}\frac{\partial}{\partial x^{i}}-\Gamma_{jk}^{i} v^{j}v^{k} \frac{\partial}{\partial v^{i}}+Y_{i}^{j}(x)v^{i}\frac{\partial}{\partial v^{j}}-g^{ij}\partial_{x^{i}} U(x) \frac{\partial}{\partial v^{j}}. \]
The generator can also be thought as a Hamiltonian or as a Lagrangian field, see \cite{mt23}*{Lemma A.1}. We the Lorentz force to a map between covectors as in \cite{dpsu}
\begin{align*}
    Y \colon T^{*}M & \to T^{*}M, \\
    u &\mapsto -Y_{i}^{j}u_{j}.
\end{align*}
Applying $G_{\MP}$ to a polynomial in $v$ of degree $1$ and we get
\begin{align*}
    G_{\MP}[u_{i}(x)v^{i}+\varphi(x)] &=G_{\mu}[u_{i}(x)v^{i}+\varphi(x)]-g^{kj} \partial_{x^{k}}U \frac{\partial}{\partial v^{j}} u_{i}v^{i}  \\
    &=(d^{s}u)_{ij}v^{i}v^{j}+(\varphi_{,j}-Y(u)_{j})v^{j}-g^{ij}u_{i}(\partial_{x^{j}}U).
\end{align*}
Assuming that this is equal to the polynomial of order $2$ as in \eqref{eq:poly}, and looking at the even and odd parts, we find
\begin{align*}
    h_{ij}v^{i}v^{j}+V&=(d^{s}u)_{ij}v^{i}v^{j}-(dU,u)_{g}, \\
    \beta_{i}&=\varphi_{,j}-Y(u)_{j}.
\end{align*}
From this, we see that 2-tensors and functions are coupled. Even if we use that $E(x,v)=k$, there is no clear way to separate them. If we force $h=d^{s}u$ and $V=-(dU,u)_{g}$, we find
\begin{equation}
    \begin{pmatrix}
    h \\ \beta \\ V
\end{pmatrix}=\mathbf{d}_{1}\begin{pmatrix}
    u \\ \varphi \\ \eta
\end{pmatrix}, \qquad \mathbf{d}_{1}:=\begin{pmatrix}
    d^{s} & 0 & 0 \\ -Y & d & 0 \\ -(dU,\cdot )_{g} & 0 & 0
\end{pmatrix},
\end{equation}
where $u$ is a covector, $\varphi$ and $\eta$ are functions. However, this operator $\dd_{1}$ does not take into account the relation between 2-tensors and functions. A better definition of potential if we one that we explore below.

\begin{rmk} \label{rmk:lost}
    From Lemma \ref{lemma:i_phi} and Proposition \ref{prop:rel_mpray_magray}, we see that the element $\mathbf{f}=\eta[-g,0,2(k-U)]$ still satisfies $I_{\mathcal{M}}\mathbf{f}=0$, but the corresponding polynomial associated to $\mathbf{f}$ is not always of the form $G_{\MP}\psi$. Indeed the equality reads
    \[ -\eta g_{ij}v^{i}v^{j}+\eta2(k-U)=(d^{s}u)_{ij}v^{i}v^{j}+(\varphi_{,j}-Y(u)_{j})v^{j}-(dU,u)_{g}. \]
    Looking at even and odd parts we find 
    \begin{align}
        -\eta g_{ij}v^{i}v^{j}+\eta2(k-U)&=(d^{s}u)_{ij}v^{i}v^{j}-(dU,u)_{g}, \label{eq:2ten} \\
        0&=-Y(u)+d\varphi.
    \end{align}
    Since $(x,v) \in S^{k}M$, equation \eqref{eq:2ten} becomes
    \[ (d^{s}u)_{ij}v^{i}v^{j}=\frac{(dU,u)_{g}}{2(k-U)}g_{ij}v^{i}v^{j}. \]
    Then, 
    \[ (d^{s}u)_{ij}=\frac{(dU,u)_{g}}{2(k-U)}g_{ij}, \]
    which can be written as
    \[\delta((2(k-U))^{\frac{n}{2}}u)=0,\]
    where $\delta$ is the divergence given by
    \[ \delta w=g^{ij}(\partial_{x^{i}}w_{j}-\Gamma_{ij}^{k}w_{k}). \]
    This phenomena seems to occurs because $G_{\MP}$ does not take into account the relation between 2-tensors and functions given by the energy.
\end{rmk}

We consider a second option. We linearize the gauge transformation group. Take two $\MP$-systems $(g,\alpha,U)$ and $(g',\alpha',U')$ that are $k$-gauge equivalent, i.e., $g'=\frac{1}{\mu}f^{*}g$, $\alpha'=f^{*}\alpha+d\varphi$, $U'=\mu (f^{*}U-k)+k$, where $f\colon M \to M$ is a diffeomorphism with $f|_{\partial M}=id_{\partial M}$, $\varphi \colon M \to \R$ is a smooth function with $\varphi|_{\partial M}=0$, and $\mu \in C^{\infty}(M)$ is a strictly positive function. We want to linearize this near $f=id_{M}$, $\varphi=0$, and $\mu=1$. So, let $f_{\tau}$ be a smooth family of such diffeomorphism, with $f_{0}=id_{M}$, let $\varphi^{\tau}$ be a smooth family of such function with $\varphi^{0}=0$, and let $\mu^{\tau}$ be a smooth family of such positive functions with $\mu^{0}=1$. Let $g^{\tau}=\frac{1}{\mu_{\tau}} f_{\tau}^{*}g$, $\alpha^{\tau}=f_{\tau}^{*}\alpha+d\varphi^{\tau}$, and $U^{\tau}=\mu^{\tau}(f_{\tau}^{*}U-k)+k$. We will compute the derivatives at $\tau=0$. It is well known (see \cite{sha_ray}*{Equation (3.1.5)}) that $\frac{d}{d\tau}|_{\tau=0}f_{\tau}^{*}g=2d^{s}u^{\flat}$, where $d^{s}u^{\flat}$ is the symmetric differential of $u^{\flat}$, and $u=\frac{df_{\tau}}{d\tau}|_{\tau=0}$. Here $\flat$ denoted the musical isomorphism $\flat \colon TM \to T^{*}M$ induced by $g$. So, 
\[ \frac{dg^{\tau}}{d\tau} \bigg|_{\tau=0}=-\frac{d \mu^{\tau}}{d\tau} \bigg|_{\tau=0}g+2d^{s} u^{\flat}. \]
Then, using the results in \cite{dpsu} we have:
\[ \frac{d\alpha^{\tau}}{d\tau} \bigg|_{\tau=0}=-Y(u^{\flat})+d((\alpha,u^{\flat})_{g}+\psi), \]
where $\psi=\frac{d\varphi^{\tau}}{d\tau}|_{\tau=0}$.
Finally, we compute
\begin{align*}
    \frac{dU^{\tau}}{d\tau} \bigg|_{\tau=0} &=\frac{d \mu^{\tau}}{d\tau}\bigg|_{\tau=0}(U-k)+\frac{d(f_{\tau}^{*}U )}{d\tau} \bigg|_{\tau=0} \\
    &=\frac{d \mu^{\tau}}{d\tau}\bigg|_{\tau=0}(U-k)+(dU,u^{\flat})_{g}.
\end{align*}
Letting $\eta:=\frac{d \mu^{\tau}}{d\tau}|_{\tau=0}$, we obtain
\[ \frac{d}{d\tau}\bigg|_{\tau=0} \left[\frac{1}{2}g^{\tau},-\alpha^{\tau},-U^{\tau} \right]=\dd_{2} [u^{\flat},-(\alpha,u^{\flat} )_{g}-\psi,\eta], \]
where
\[ \dd_{2}=\begin{pmatrix}
    d^{s} & 0 & -\frac{1}{2}g \\ -Y & d & 0 \\ -(dU,\cdot )_{g} & 0 & k-U
\end{pmatrix}. \]

What makes $\mathbf{d}_{1}$ different from $\mathbf{d}_{2}$ are the terms that appear on the last column of $\mathbf{d}_{2}$. This last row is what to take account of elements in $\ker \Phi$, see Lemma \ref{lemma:Phi} and Remark \ref{rmk:lost}. 

\begin{lemma}
Let $\mathbf{f}=[h,\beta,V] \in \mathbf{L}^{2}(M)$, and $\mathbf{w}=[v,\varphi,\eta]$, where $[u,\varphi,0] \in \mathcal{H}_{0}^{1}(M)$, $\eta \in L^{2}(M)$ and $\mathbf{f}_{i}=\mathbf{d}_{i}\mathbf{w}$ with $i \in \{1,2\}$. Then $I\mathbf{f}=0$.
\end{lemma}

\begin{proof}
If $\mathbf{f}_{1}=\mathbf{d}_{1}\mathbf{w}$, then $\mathbf{f}_{1}=G_{\MP}\psi$, where $\psi$ is a polynomial vanishing on the boundary. Then, 
\[ I\mathbf{f}_{1}=\psi|_{\partial M}=0. \]

On the other hand, if $\mathbf{f}_{2}=\mathbf{d}_{2}\mathbf{w}$, then
\[ \mathbf{f}_{2}=\mathbf{d}_{2}\mathbf{w}=\mathbf{d}_{1}\mathbf{w}+\eta \begin{pmatrix}
    -\frac{1}{2}g \\ 0 \\ k-U
\end{pmatrix}=:\mathbf{d}_{1}\mathbf{w}+\eta \mathbf{r}. \]
By the previous part, $I(\dd_{1}\mathbf{w})=0$. By Lemma \ref{lemma:i_phi}, we also have $I(\eta \mathbf{r})=0$. Therefore, $I(\mathbf{f}_{2})=0$ as well.
\end{proof}

Since $\dd_{2}$ is more general, we define the elements on its image as potentials.

\begin{defin}
    We say that an element $\mathbf{f}=[h,\beta,V] \in \LL^{2}(M)$ is a \emph{potential} if $\mathbf{f}=\dd \mathbf{w}$, where $\mathbf{w}=[u,\varphi,\eta]$, $[u,\varphi,0] \in \mathcal{H}_{0}^{1}(M)$, $\eta \in L^{2}(M)$ and $\dd:=\dd_{2}$.
\end{defin}

\begin{defin} \label{defin:s_inj}
    We say that $I$ is s-injective if $I\mathbf{f}=0$ implies that $\mathbf{f}$ is a potential.
\end{defin}

\subsection{Relation between MP-potentials and magnetic ones}

Recall that a potential for $I_{\mathcal{M}}$ is a pair such that 
\[ \begin{pmatrix}
    h \\ \beta
\end{pmatrix}=\dd_{\mathcal{M}} \begin{pmatrix}
    u \\ \varphi
\end{pmatrix},  \]
where 
\[ \dd_{\mathcal{M}}=\begin{pmatrix}
    d_{G}^{s} & 0 \\
    -Y_{G} & d
\end{pmatrix}. \]
Here $d_{G}^{s}$ is the symmetric differential with respect to the metric $G$, and $Y_{G}$ is the Lorentz force of $\Omega=d\alpha$ with respect to $G$. In this subsection we will find a relation between magnetic potentials, i.e., elements on the range of $\dd_{\mathcal{M}}$, and potentials for the $\MP$-ray transform. First, we will obtain a more explicit formula for $\dd_{\mathcal{M}}$ in terms of $Y_{g}$ and $d_{g}^{s}$ (here $Y_{g}$ is the Lorentz force of $\alpha$ with respect to $g$, while $d_{g}^{s}$ is the symmetric differential with respect to $g$). Using the definition of the Lorentz force, we obtain for any $v, w \in T_{x}M$
\[ (Y_{x,g}(v), w)_{g}=\Omega_{x}(v,w)=(Y_{x,G}(v),w)_{G}=2(k-U) (Y_{x,G}(v),w )_{g}.  \]
Hence, $Y_{G}=\frac{1}{2(k-U)}Y_{g}$. Now we would like to find $d_{G}^{s}$ in terms of $d_{g}^{s}$. Let $^{G}\nabla$ and $^{g}\nabla$ be the Levi-Civita connections with respect to the metrics $G$ and $g$, respectively. We use a similar notation for the Christoffel's symbols. Recall that these operators act on 1-tensors. Since
\[ (d_{G}^{s}u)_{ij}=\frac{1}{2}(^{G}\nabla_{i}u_{j}+^{G}\nabla_{j}u_{i}), \]
we need to write $^{G}\nabla$ in terms of $^{g}\nabla$. First, we write the Christoffel symbol with respect to $G$, $^{G}\Gamma_{ij}^{k}$, in terms of the Christoffel symbols with respect to $g$, $^{g}\Gamma_{ij}^{k}$. \cite{lee}*{Proposition 7.29} yields
\begin{align*}
    ^{G}\Gamma_{ij}^{k} &=^{g}\Gamma_{ij}^{k}+\frac{1}{2}\frac{\partial \log 2(k-U)}{\partial x^{i}} \delta_{j}^{k}+\frac{1}{2}\frac{\partial \log 2(k-U)}{\partial x^{j}} \delta_{i}^{k}-g^{k \ell} \frac{1}{2}\frac{\partial \log 2(k-U)}{\partial x^{\ell}} g_{ij} \\
    &=^{g}\Gamma_{ij}^{k}-\frac{1}{2(k-U)} \left( \frac{\partial U}{\partial x^{i}} \delta_{j}^{k}+\frac{\partial U}{\partial x^{j}} \delta_{i}^{k}-g^{k \ell} \frac{\partial U}{\partial x^{\ell}} g_{ij} \right).
\end{align*}
Then, 
\begin{align*}
    (d_{G}^{s}u)_{ij} &= \frac{1}{2} \left( \frac{\partial u_{j}}{\partial x^{i}}-u_{k} ^{G}\Gamma_{ij}^{k} +\frac{\partial u_{i}}{\partial x^{j}}-u_{k} ^{G}\Gamma_{ij}^{k} \right) \\
    &= \frac{1}{2} (^{g}\nabla_{i}u_{j}+^{g}\nabla_{j}u_{i}) + \frac{u_{k}}{2(k-U)} \left( \frac{\partial U}{\partial x^{i}} \delta_{j}^{k}+\frac{\partial U}{\partial x^{j}} \delta_{i}^{k}-g^{k \ell} \frac{\partial U}{\partial x^{\ell}} g_{ij} \right) \\
    &= \frac{1}{2} (^{g}\nabla_{i}u_{j}+^{g}\nabla_{j}u_{i}) + \frac{1}{2(k-U)} \left( \frac{\partial U}{\partial x^{i}} u_{j}+\frac{\partial U}{\partial x^{j}} u_{i} \right) -\frac{1}{2(k-U)} (dU,u)_{g} g_{ij} \\
    &=2(k-U)d_{g}^{s} \left( \frac{u}{2(k-U)} \right)_{ij}-\left( dU,\frac{u}{2(k-U)} \right)_{g}g_{ij}.
\end{align*}
Hence, 
\[ \dd_{\mathcal{M}} \begin{pmatrix}
    u \\ \varphi
\end{pmatrix}=\begin{pmatrix}
    2(k-U)d_{g}^{s}-(dU,\cdot )_{g}g & 0 \\ -Y_{g} & d
\end{pmatrix} \begin{pmatrix}
    \frac{u}{2(k-U)} \\ \varphi
\end{pmatrix}. \]

Similarly as we did the previous section, we can express the relation between $\dd$ and $\dd_{\mathcal{M}}$ with a map (explicitly given by Proposition \ref{prop:com_d} above), which we define as follows
\begin{align*}
    \phi \colon \mathcal{H}_{0}^{1}(M) & \to \mathcal{H}_{0,\mathcal{M}}^{1}(M), \\
    [u,\varphi,\eta] & \mapsto [2(k-U) u,\varphi ].
\end{align*}
Since $U \in C^{\infty}(M)$ and $M$ is compact, this map is well defined. 

\begin{lemma} \label{lemma:ker_phi}
    $\phi$ is a surjective linear map with kernel
    \[ \ker \phi=\{ [0,0,\eta] \in \mathcal{H}_{0}^{1} \}. \]
\end{lemma}

\begin{proof}
    It is clear that $\phi$ is linear, and given $\mathbf{w}=[u,\varphi] \in \mathcal{H}_{0,\mathcal{M}}^{1}(M)$, we have that $\phi([\frac{u}{2(k-U)},\varphi,\eta])=\mathbf{w}$, where $\eta \in H_{0}^{1}(M)$. 
    
    Now take $[u,\varphi,\eta] \in \ker \phi$. Then $u=0$ and $\varphi=0$. Clearly, elements of the form $[0,0,\eta]$ are part of $\ker \phi$. This proves the lemma.
\end{proof}

\begin{rmk}
    We could define $\phi$ in a more general way from $\cal^{2}(M)$ to $\cal_{\mathcal{M}}^{2}(M)$ and obtain an analog to Lemma \ref{lemma:ker_phi}. However, defining it between the Sobolev spaces is more suitable for our proposes, see Proposition \ref{prop:com_d} below.
\end{rmk}

Using $\Phi$ and $\phi$, we can fully understand the relation between $\dd$ and $\dd_{\mathcal{M}}$. Indeed, we have the following result

\begin{prop} \label{prop:com_d}
    The following diagram
    \[\begin{tikzcd}
	{\mathcal{H}_{0}^{1}(M)} && {\mathbf{L}^{2}(M)} \\
	& \\
	{\mathcal{H}_{0,\mathcal{M}}^{1}(M)} && {\mathbf{L}_{\mathcal{M}}^{2}(M)}
	\arrow["{\mathbf{d}}", from=1-1, to=1-3]
	\arrow["\phi"', from=1-1, to=3-1]
	\arrow["\Phi", from=1-3, to=3-3]
	\arrow["{\mathbf{d}_{\mathcal{M}}}"', from=3-1, to=3-3]
    \end{tikzcd}\]
    commutes.
\end{prop}

\begin{proof}
    Let $\mathbf{w}=[u,\varphi,\eta] \in \cal^{2}(M)$. We have to prove that $\Phi \dd \mathbf{w}=\dd_{\mathcal{M}}\phi \mathbf{w}$. We do the explicit computation. First
    \[ \Phi \dd \mathbf{w}=\Phi \begin{pmatrix}
        d_{g}^{s}u -\frac{\eta}{2}g \\ -Y_{g}(u)+ d\varphi \\ -(dU,u)_{g}+\eta(k-U).
    \end{pmatrix}=\begin{pmatrix}
        2(k-U)d_{g}^{s}u-(dU,u)_{g}g \\ -Y_{g}(u)+d\varphi
    \end{pmatrix}. \]
    On the other hand, we have
    \[ \dd_{\mathcal{M}}\phi \mathbf{w}=\dd_{\mathcal{M}} \begin{pmatrix}
        2(k-U)u \\ \varphi
    \end{pmatrix}=\begin{pmatrix}
        2(k-U)d_{g}^{s}u-(dU,u)_{g}g  \\ -Y_{g}(u)+d\varphi
    \end{pmatrix}. \]
    This proves the equality.
\end{proof}

Finally, we prove the following result which gives a relation between the notions of s-injectivity for $I$ and $I_{\mathcal{M}}$.

\begin{prop} \label{prop:s_inj}
    Let $(g,\alpha,U)$ be a simple $\MP$-system of energy $k$, and let $(G,\alpha)$ be its magnetic reduction. Then, $I_{\mathcal{M}}$ is s-injective if and only if $I$ is s-injective.
\end{prop}

\begin{proof}
    Let $\mathbf{f}=[h,\beta,V] \in \LL^{2}(M)$. Note that $(G,\alpha)$ is again simple by Lemma \ref{lemma:basics_mp}. By Proposition \ref{prop:rel_mpray_magray}, $I_{\mathcal{M}}\Phi(\mathbf{f})=0$. Since $I_{\mathcal{M}}$ is s-injective there exists a pair $[u,\varphi]$ such that $\Phi \mathbf{f}=\dd_{\mathcal{M}}[u,\varphi]$. Take any $\eta_{1} \in H_{0}^{1}(M)$. Then, $\phi \mathbf{w}_{1}=[u,\varphi]$, where $\mathbf{w}_{1}=[\frac{u}{2(k-U)},\varphi,\eta_{1}]$. Thus, $\Phi \mathbf{f}=\dd_{\mathcal{M}} \phi \mathbf{w}_{1}$. By Proposition \ref{prop:com_d}, we have $\dd_{\mathcal{M}} \phi \mathbf{w}_{1}=\Phi \dd \mathbf{w}_{1}$. Hence, $\Phi \mathbf{f}=\Phi \dd \mathbf{w}_{1}$. So, $\mathbf{f}-\dd \mathbf{w}_{1} \in \ker \Phi$. In virtue of Lemma \ref{lemma:Phi}, there exists $\eta_{2} \in L^{2}(M)$ such that 
    \[ \mathbf{f}=\dd \mathbf{w}_{1}+\eta_{2}\begin{pmatrix}
    -g/2 \\ 0 \\ k-U \end{pmatrix}. \]
    Therefore, $\mathbf{f}=\dd\mathbf{w}$, where $\mathbf{w}=\mathbf{w}_{1}+\eta_{2}[-g/2,0,k-U]$, which shows that $I$ is s-injective.

    Reciprocally, take $\mathbf{f}_{\mathcal{M}} \in \LL_{\mathcal{M}}^{2}(M)$ with $I_{\mathcal{M}}\mathbf{f}_{\mathcal{M}}=0$. We have to show that $\mathbf{f}_{\mathcal{M}}$ is a potential in the magnetic sense. By Lemma \ref{lemma:Phi}, there exists $\mathbf{f} \in \LL^{2}(M)$ such that $\Phi \mathbf{f}=\mathbf{f}_{\mathcal{M}}$. Then Proposition \ref{prop:rel_mpray_magray} gives
    \[I\mathbf{f}(x,v)=I_{\mathcal{M}}\Phi(\mathbf{f})\left( x, \frac{v}{2(k-U(x))} \right)=I_{\mathcal{M}}\mathbf{f}_{\mathcal{M}}\left( x, \frac{v}{2(k-U(x))} \right)=0.\]
    Since $I$ is s-injective, this implies that $\mathbf{f}=\dd \mathbf{w}$. In virtue of Proposition \ref{prop:com_d}, we obtain that 
\[ \mathbf{f}_{\mathcal{M}}=\Phi \mathbf{f}=\Phi \dd \mathbf{w}=\dd_{\mathcal{M}}\phi \mathbf{w}, \]
which shows that $I_{\mathcal{M}}$ is s-injective.
\end{proof}

\section{s-injectivity of the ray transform for analytic systems} \label{sec:s_inj_an}

In this section we prove that the $\MP$-ray transform is s-injectivity for analytic $\MP$-system. 

As in \cite{dpsu}, we work on a real analytic manifold $M$, with smooth boundary $\partial M$ that does not need to be analytic. We say that $\mathbf{f}$ is analytic in the set $X$, not necessarily open, if $f$ is analytic in a neighborhood of $X$.

\begin{thm} \label{thm:an_inj}
    If $M$ is a real analytic compact manifold with boundary, and $(g, \alpha, U)$ is a real analytic simple $\MP$-system, then $I$ is s-injective.
\end{thm}

The proof is just a combination of previous results.

\begin{proof}
     Since $(g,\alpha,U)$ is a simple analytic $\MP$-system, then its reduction $(2(k-U)g,\alpha)$ is again simple by Lemma \ref{lemma:basics_mp}, and is clearly analytic. Then, the result follows from Proposition \ref{prop:s_inj}.
\end{proof}

\section{The linear problem} \label{sec:linear}

The linear problems asks to what extent the $\MP$-ray transform is s-injective. For 1-tensors and functions, we obtain an affirmative answer

\begin{thm} \label{thm:linear_1tensors}
    Let $(M, g, \alpha, U)$ be a simple $\MP$-system, $u$ a square integrable 1-form, and $\varphi$ a square integrable function on $M$. If the $\MP$-ray transform of the function $F(x,v)=\beta_i(x) v^i+V$ vanishes, then $V=0$ and $\beta=d \varphi$ for some function $\varphi \in H_0^1(M)$.
\end{thm}

\begin{proof}
Consider the magnetic system $(G,\alpha)$, where $G=2(k-U)g$. By Proposition \ref{prop:rel_mpray_magray}, we obtain 
\[ 0=I[0,\beta,V](x,2(k-U(x))v )=I_{\mathcal{M}}[gV,\beta] (x,v). \]
Let $(x,v) \in \partial_{+}S^{G}M$, and let $\gamma$ be the unique magnetic geodesic of energy $\frac{1}{2}$ the system $(G,\alpha)$ with $(\gamma(0),\dot{\gamma}(0))=(x,v)$. Hence, 
\begin{align*}
    0&=(I_{\mathcal{M}}[gV,\beta])(x,v) \\
    &=\int_{0}^{\tau^{G}(x,v)} \left \lbrace V(\gamma(t))g_{ij}(\gamma(t))\dot{\gamma}^{i}(t)\dot{\gamma}^{j}(t)+\beta_{i}(\gamma(t))\dot{\gamma}^{i}(t) \right\rbrace dt \\
    &=\int_{0}^{\tau^{G}(x,v)} \left \lbrace \frac{V(\gamma(t))}{2(k-U(\gamma(t)))}+\beta_{i}(\gamma(t))\dot{\gamma}^{i}(t) \right\rbrace dt, \\
\end{align*}
where $\tau_{G}$ is the exit time for the magnetic system $(G,\alpha)$. Then, the magnetic ray transform of the function $\Tilde{F}(x,v)=\beta_{i}v^{i}+\frac{V(x)}{2(k-U(x))}$ vanishes. In virtue of \cite{dpsu}*{Theorem 5.3}, we obtain that $\frac{V}{2(k-U)}=0$ and $\beta=d\varphi$, for some function $\varphi \in H_{0}^{1}(M)$, which proves the result.
\end{proof}

We solve the linear problem for 2-tensors assuming an extra condition, which arises because we reduce the problem to the one on magnetic systems.

\begin{thm} \label{thm:linear_2tensors}
    If $(M, g, \alpha,U)$ is a simple $\MP$-system with $k(M, 2(k-U)g, \alpha) \leq 4$, then $I_{\MP}$ is s-injective.
\end{thm}

Here $k(M,2(k-U)g,\alpha)$ is defined as follows. First, let us write $G=2(k-U)g$. Set
\begin{equation} \label{eq:k_mu}
    k_\mu(x, v)=\sup _w \{2 K_{G}(x, \sigma_{v, w})+(Y_{G}(w), v)_{G}^2+(n+3)|Y_{G}(w)|_{G}^2-2((\nabla^{G}_w Y_{G})(v), w)_{G} \},
\end{equation}
where the supremum is taken over all unit vectors $w \in T_x M$ orthogonal to $v$, and $K_{G}(x, \sigma_{v,w})$ is the sectional curvature of the 2-plane $\sigma_{v,w}$ spanned by $v$ and $w$. Define
\[ k_\mu^{+}(x, v)=\max \{0, k(x,v)\} \]
and
\[ k(M, 2(k-U)g, \alpha)=\sup _\gamma T_\gamma \int_0^{T_\gamma} k_\mu^{+}(\gamma(t), \dot{\gamma}(t)) d t,\]
where the supremum is taken over all unit speed magnetic geodesics $\gamma \colon [0, T_\gamma ] \to M$ running between boundary points. Obviously, definition of $k_{\mu}$ can be done more explicit in terms of $g$, $Y_{g}$ and $U$, but we omit that in this work. However, we would like to mention that by expanding $K_{G}$ in terms of $K_{g}$ and derivatives of $U$, we see that the condition in Theorem \ref{thm:linear_2tensors} holds when $(M,g)$ is negatively curved, the $C^{1}$ norm of $Y$ is small enough, and the $C^{2}$ norm of $U$ is small enough. This condition is always true in a small enough subset of $M$, with $(M,g)$ negatively curved. 

Finally, we would like to mention that quantities as in \eqref{eq:k_mu} are ``classical'', in the sense that they appear in the problem of inverting the ray transform for tensors, see \cite{sha}. In our case, the reason is similar, and this function appear as a consequence of our methods, since we use the results in \cite{dpsu}, which follows the strategy presented in \cite{sha}.

\begin{proof}[Proof of Theorem \ref{thm:linear_2tensors}]
The proof is similar to the one of Theorem \ref{thm:an_inj}. Consider the magnetic system $(G,\alpha)$, where $G=2(k-U)g$. In light of \cite{dpsu}*{Theorem 5.4} we have that $I_{\mathcal{M}}$ is s-injective. Now the result follows from Proposition \ref{prop:s_inj}.
\end{proof}

\section{Generic results} \label{sec:generic}

In this section we show that for large $m$, the set of $\MP$-systems for which the $\MP$-ray transform is s-injective is open and dense in the $C^{m}$ topology. We also prove that two $\MP$-system that are close to a system in this set, and that have the same boundary action function at energy level $k$, the systems must be $k$-gauge equivalent.

\subsection{Generic s-injectivity} \label{subsec:gen_s_inj}

\begin{defin}
    For a fixed manifold $M$, we define $\mathcal{G}^{m}$ to be the set of simple $C^m$ systems $(g, \alpha,U)$ with s-injective $\MP$-ray transform $I$.
\end{defin}

There is an analogue of $\mathcal{G}^{m}$ for magnetic systems (see \cite{dpsu}*{Definition 4.10}), which we denote by $\mathcal{G}_{\mathcal{M}}^{m}$.

\begin{rmk} \label{rmk:s_inj}
    Note that Proposition \ref{prop:s_inj} can be seen in the following way: $(g,\alpha,U) \in \mathcal{G}^{m}$ if and only if $(2(k-U)g,\alpha) \in \mathcal{G}_{\mathcal{M}}^{m}$.
\end{rmk}

The objective is to obtain the following generalization of \cite{dpsu}*{Theorem 4.11}.

\begin{thm} \label{thm:gen_sinj}
    There exists $m_0>0$, such that for $m \geq m_0$, the set $\mathcal{G}^m$ is open and dense in the set of all simple $C^m$ $\MP$-systems $(g, \alpha,U)$ and contains all real analytic simple $\MP$-systems.
\end{thm}

First we prove that $\mathcal{G}^{m}$ is open.

\begin{lemma} \label{lemma:gen_sinj}
    The set $\mathcal{G}^{m}$ is open in the $C^m$ topology in the set of all simple $C^{m}$ $\MP$-systems, provided $m \gg 1$.
\end{lemma}

\begin{proof}
    Take $(g,\alpha,U)\in \mathcal{G}^{m}$, and let $\epsilon>0$ to be chosen. Let $(g',\alpha',U')$ be a simple $\MP$-system which is $\eps$-close to $(g,\alpha,U)$, that is,
    \[ \|g-g'\|_{C^m(M)}+\|\alpha-\alpha'\|_{C^m(M)}+\|U-U'\|_{C^{m}(M)}\leq \eps \]
    Consider its magnetic reduced systems $(G,\alpha)$, $(G',\alpha)$ with $G=2(k-U)g$, $G'=2(k-U)$. Then, 
    \begin{align*}
        \|G-G'\|_{C^{m}} &=\| 2(k-U)g-2(k-U')g' \|_{C^{m}} \\
        & \leq 2k \|g-g'\|_{C^{m}}+\|Ug-U'g'\|_{C^{m}} \\
        & \leq 2k \|g-g'\|_{C^{m}}+\|Ug-U'g\|_{C^{m}}+\|U'g-U'g'\|_{C^{m}} \\
        & \leq 2k \|g-g'\|_{C^{m}}+\|g\|_{C^{m}}\|U-U'\|_{C^{m}}+\|U'\|_{C^{m}}\|g-g'\|_{C^{m}} \\
        & \leq 2k\eps +\|g\|_{C^{m}}\eps+\|U'\|_{C^{m}} \eps \\
        & \leq \eps (2k+\|g\|_{C^{m}}+\|U\|_{C^{m}}+\eps).
    \end{align*}
    Note that by Proposition \ref{prop:s_inj} (see also Remark \ref{rmk:s_inj}), we have that $(G,\alpha) \in \mathcal{G}_{\mathcal{M}}^{m}$. Since this last set is open in the $C^{m}$ topology (\cite{dpsu}*{Corollary 4.4}), we can choose $\eps$ small enough and $m$ big enough so that $(G',\alpha') \in \mathcal{G}_{\mathcal{M}}^{m}$. Therefore, again by Proposition \ref{prop:s_inj}, we conclude that $(g',\alpha',U') \in \mathcal{G}^{m}$. This shows that a ball of radius $\eps$ centered in $(g,\alpha,U) \in \mathcal{G}^{m}$ is a subset of $\mathcal{G}^{m}$, finishing the proof.
\end{proof}

\begin{proof}[Proof of Theorem \ref{thm:gen_sinj}]
    By Lemma \ref{lemma:gen_sinj}, we obtain that set $\mathcal{G}^{m}$ is open. By Theorem \ref{thm:an_inj}, we obtain that $\mathcal{G}^{m}$ contains all real analytic simple $\MP$-systems, and therefore, is dense.
\end{proof}

\subsection{Generic local boundary rigidity} \label{subsec:gen_rig}

We will prove that near each $(g_0, \alpha_0 , U)$ in the generic set $\mathcal{G}^{m}$ , the boundary action function determines the $\MP$-system.

\begin{thm} 
    Let $m_{0}$ be as in Theorem \ref{thm:gen_sinj}. There exists $m \geq m_0$ such that for every $(g_0, \alpha_0, U_{0}) \in \mathcal{G}^m$, there is $\eps>0$ such that for any two $\MP$-systems $(g, \alpha, U)$, $(g, \alpha,U)$ with
    \begin{align*}
        \|g-g_0\|_{C^m(M)}+\|\alpha-\alpha_{0}\|_{C^m(M)}+\|U-U_{0}\|_{C^{m}(M)} &\leq \eps, \\
        \|g'-g_0\|_{C^m(M)}+\|\alpha'-\alpha_{0}\|_{C^m(M)}+\|U'-U_{0}\|_{C^{m}(M)} &\leq \eps,
    \end{align*}
    we have the following:
    \[\mathbb{A}_{g, \alpha,U}=\mathbb{A}_{g', \alpha',U'} \quad \text { on } \partial M \times \partial M, \]
    (where the boundary action functions are of energy $k$), implies that $(g, \alpha,U)$ and $(g', \alpha',U')$ are $k$-gauge equivalent.
\end{thm}

\begin{proof}
    Consider the reduction of the $\MP$-systems, that is, we consider $(G_{0},\alpha_{0})$, $(G,\alpha)$ and $(G,\alpha')$, where $G_{0}=2(k-U_{0})g_{0}$, $G=2(k-U)g$ and $G'=2(k-U')g'$. We would like to apply \cite{dpsu}*{Theorem 6.5} to these systems in order to obtain our result. First we show that $G_{0}$ and $G$ are close in the $C^{m}$ topology. Indeed, 
\begin{align*}
    \|G_{0}-G \|_{C^{m}(M)}  \leq & 2k\| g_{0}-g\|_{C^{m}(M)}+2 \| U_{0}g_{0}-Ug \|_{C^{m}(M)} \\
     \leq & 2k\| g_{0}-g\|_{C^{m}(M)}+2c(m)\| U\|_{C^{m}(M)} \| g_{0}-g\|_{C^{m}(M)} \\
    &+2c(m) \|g\|_{C^{m}(M)} \| U_{0}-U \|_{C^{m}(M)} \\
    \leq & \max \{2k,2c(m) \| U\|_{C^{m}(M)}, 2c(m) \|g\|_{C^{m}(M)} \}2\eps.
\end{align*}
A similar computation shows that $G'$ is close to $G$ in the $C^{m}$ topology. This imply that $(G,\alpha)$ and $(G',\alpha')$ are close to $(G_{0},\alpha_{0})$. The hypothesis about the boundary action functions of the $\MP$-systems imply, in virtue of Lemma \ref{lemma:basics_mp}, that the boundary action functions $\A$ and $\A'$, of the corresponding magnetic systems $(G,\alpha)$ and $(G',\alpha')$, coincide. Finally, since $(g_{0},\alpha_{0},U_{0}) \in \mathcal{G}^{m}$, by Proposition \ref{prop:s_inj} (see also Remark \ref{rmk:s_inj}), we have $(G_{0},\alpha_{0}) \in \mathcal{G}_{\mathcal{M}}^{m}$. This allow us to invoke \cite{dpsu}*{Theorem 6.5}, which gives the existence of a $C^{m+1}$ diffeomorphism $f \colon M \to M$ fixing the boundary, and a function $\varphi \in C^{m+1}(M)$ vanishing on the boundary, so that, $G'=f^{*}G$ and $\alpha'=f^{*}\alpha+d\varphi$. The conclusion now follows from Lemma \ref{lemma:mp_red}.
\end{proof}

\begin{rmk}
    We can characterize the generic set in another way. Let
    \[ \tilde{\mathcal{G}}^{m}=\left\lbrace \left( \frac{1}{2(k-U)}G,\alpha,U \right): (G,\alpha) \in \mathcal{G}_{\mathcal{M}}^{m}, \, U \in C^{\infty}(M) \right\rbrace. \]
    Then $\tilde{\mathcal{G}}^{m}=\mathcal{G}^{m}$. Indeed, take $(g,\alpha,U) \in \tilde{\mathcal{G}}^{m}$. Then $2(k-U)g=G$ and $(G,\alpha) \in \mathcal{G}_{\mathcal{M}}^{\alpha}$. By Proposition \ref{prop:s_inj} (see also Remark \ref{rmk:s_inj}) we conclude that $(g,\alpha,U) \in \mathcal{G}^{m}$. On the other hand, given $(g,\alpha,U) \in \mathcal{G}^{m}$, again by Proposition \ref{prop:s_inj} we obtain $(G,\alpha) \in \mathcal{G}_{\mathcal{M}}^{m}$, where $G=2(k-U)g$. Then, by definition, $(g,\alpha,U)=(\frac{1}{2(k-U)}G,\alpha,U) \in \tilde{\mathcal{G}}^{m}$.
\end{rmk}

It would be interesting to study the $\MP$-ray transform as in \cite{su05}, \cite{dpsu}. In particular, to obtain a potential-solenoidal decomposition. After this, one could try to obtain a stability estimate for the solenoidal parts, involving the normal operator, generalizing the known result in the magnetic case. The study of the microlocal properties of the normal operator it would be interesting as well.

\appendix

\section{Santal\'o's formula} \label{sec:app}

Here we prove a Santal\'o's formula for $\MP$-systems. The usual proof of this result in the Riemannian case is based in the fact that the Liouville measure in $SM$ is invariant over the geodesic flow, and the application Stokes' theorem (\cite{psu}). One can also obtain it by using a more geometric/dynamical approach (\cite{sha_ray}). We take another path here, similar to the one in \cite{lz}. We take advantage that Santal\'o's formula is already known in the magnetic case \cite{dpsu}, and we use the relation between an $\MP$-systems and its magnetic reduction to obtain the result.

Here $d\Sigma_{k}^{2n-1}$ is the Liouville measure in $S^{k}M$, while $d\mu_{k}$ is the measure in $\partial_{+}S^{k}M$ given by
\[ d\mu_{k}(x,v)=(v,\nu_{k}(x))_{g}d\Sigma_{k}^{2n-2}(x,v), \]
where $\nu_{k}(x)$ is the inward normal vector to $\partial M$ at $x$ with $|\nu_{k}(x)|_{g}=2(k-U(x))$, and $d\Sigma_{k}^{2n-2}$ is the Liouville measure on $\partial_{+}S^{k}M$. We denote the unit sphere bundle of $M$ with respect to the metric $G$ by $S^{G}M$. We also write
\[ \partial_{+}S^{G}M=\{(x, w) \in S_{x}^{G}M: x \in \partial M, (w, \nu_{G}(x))_{G} \geq 0\}, \]
where $\nu_{G}(x)$ is the inward normal vector to $\partial M$ at $x$ with $|\nu_{G}(x)|_{G}=1$. Let $d\sigma_{G}$, $d\sigma_{k}$ denote the measures on $S_{x}^{G}M$ and $S_{x}^{k}M$, respectively. Let $d\mu_{G}$ denote the measure in $\partial_{+}S^{G}M$ given by
\[ d\mu_{G}(x,w)=(w,\nu_{G}(x))_{G}d\Sigma_{G}^{2n-2}(x,w), \]
where $d\Sigma_{G}^{2n-2}$ is the Liouville measure on $\partial_{+}S^{G}M$. Finally, let $d\Sigma_{G}^{2n-1}$ denote the Liouville measure on $S^{G}M$. As in Section \ref{sec:ray_transform}, we write $P(x)=2(k-U(x))$.

\begin{prop}[Santal\'o's formula] \label{prop:santalo}
Let $(M,g,\alpha,U)$ be a simple $\MP$-system of energy $k$. Then for any continuous function $f \colon S^{k}M \to \R$ we have
\[ \int_{S^{k}M} f d\Sigma_{k}^{2n-1}=\int_{\partial_{+}S^{k}M} \left( \int_{0}^{\tau(x,v)} P(\sigma_{x,v}(t))^{\frac{1}{2}} f(\sigma_{x,v}(t),\dot{\sigma}_{x,v}(t)) dt\right) P(x)^{-1}d\mu_{k}. \]
\end{prop}

\begin{proof}
Consider the reduced magnetic system $(G,\alpha)$, with $G=2(k-U)\alpha$. By Lemma \ref{lemma:basics_mp}, $\gamma(s)=\sigma(t(s))$ is a unit speed magnetic geodesic, where
\[ s(t)=\int_{0}^{t}2(k-U(\sigma)). \]
Then,
\[ ds=P(x)dt, \quad \frac{d\gamma}{ds}=\frac{d \sigma}{dt}P^{-1}(x). \]
Note that the volume forms corresponding to $G$ and $g$ are related by
\[ d\vol_{G}^{n}=P^{\frac{n}{2}}\vol_{g}^{n}, \quad d\vol_{G}^{n-1}=P^{\frac{n-1}{2}}\vol_{g}^{n-1}. \]
Let $v \in S_{x}^{k}M$ and define $w=vP^{-1}$. Then, 
\[  v \in S_{x}^{k}M \iff w \in S^{G}M. \]
The measure in $S_{x}^{G}M$ is given, in local coordinates, by
\[ d\sigma^{G}=\sqrt{\det G(x)} \sum_{i=1}^{n}(-1)^{i-1} w^{i} dw^{1} \wedge \cdots \wedge \hhat{dw^{i}} \wedge \cdots \wedge dw^{n}. \]
Since $w=vP^{-1}$, we find that the measure in $S_{x}^{k}M$ is given by
\begin{align*}
    d\sigma^{k}&=\sqrt{\det g(x)} \sum_{i=1}^{n}(-1)^{i-1} \frac{v^{i}}{|v|_{g}} dv^{1} \wedge \cdots \wedge \hhat{dv^{i}} \wedge \cdots \wedge dv^{n} \\
    &=\sqrt{\det g(x)}P^{n-\frac{1}{2}}\sum_{i=1}^{n}(-1)^{i-1} w^{i} dw^{1} \wedge \cdots \wedge \hhat{dw^{i}} \wedge \cdots \wedge dw^{n} \\
    &=P^{\frac{n-1}{2}}d\sigma^{G}.
\end{align*}
Then, 
\[ d\vol_{g}^{n} d\sigma^{k}=P^{-\frac{1}{2}} d\vol_{G}^{n} d\sigma^{G}.\]
in other words, the Liouville forms are related in the following way
\[ d\Sigma_{k}^{2n-1}=P^{-\frac{1}{2}}d\Sigma_{G}^{2n-1}. \]
Using the change of variables $w=P^{-1}v$, Santal\'o's formula in the magnetic case (\cite{dpsu}*{Lemma A.8}) gives
\begin{align*}
    &\int_{S^{k}M}f(x,v) d\Sigma_{k}^{2n-1}(x,v) \\
    =&\int_{S^{G}M}f(x,Pw)P^{-\frac{1}{2}}(x)d\Sigma_{G}^{2n-1}(x,w) \\
    =&\int_{\partial_{+}S^{G}M} \left( \int_{0}^{\tau^{G}(x,w)}  f(\gamma(s),P(\gamma(s))\dot{\gamma}(s))P^{-\frac{1}{2}}(\gamma(s))ds \right)  d\mu_{G}(x,w) \\
    =&:\mathcal{I},
\end{align*} 
where $\tau^{G}$ is the exit function for the magnetic system $(G,\alpha)$. Here $\dot{\gamma}=\frac{d\gamma}{ds}$. Since $d\Sigma_{G}^{2n-2}=d\Sigma_{k}^{2n-2}$, we find
\[ d\mu_{G}(x,w)=P(P^{-1}v,P^{-1}\nu_{k})_{g}d\Sigma_{k}^{2n-2}(x,v)=P^{-1}(x)d\mu_{k}(x,v), \]
where $\nu_{k}(x)=P^{\frac{1}{2}}\nu(x)$.
Therefore, 
\[ \mathcal{I}=\int_{\partial_{+}S^{k}M} \left(\int_{0}^{\tau(x,v)} f(\sigma(t),\dot{\sigma}(t)) P^{\frac{1}{2}}(\sigma(t))dt \right)  P^{-1}(x)d \mu_{k}, \]   
where $\dot{\sigma}=\frac{d\sigma}{dt}$.
\end{proof}

\end{document}